\documentclass[english,12pt,oneside]{amsproc}
\usepackage[english]{babel}
\usepackage{a4wide}
\usepackage{amsthm}
\usepackage{graphics}
\usepackage{amsfonts, amssymb, amscd, amsmath}
\usepackage{latexsym}
\usepackage[matrix,arrow,curve]{xy}
\usepackage{mathabx}
\usepackage{color}
\usepackage{pbox}
\usepackage{tikz}
\usetikzlibrary{matrix,decorations.pathreplacing,positioning}
\usepackage{hyperref}

\DeclareMathOperator{\Cone}{Cone} 
\DeclareMathOperator{\Ker}{Ker}

\DeclareMathOperator{\Hom}{Hom} 
\DeclareMathOperator{\odd}{odd} \DeclareMathOperator{\SU}{SU}
\DeclareMathOperator{\Pe}{Pe} \DeclareMathOperator{\relint}{rel int}
\DeclareMathOperator{\conv}{convhull} \DeclareMathOperator{\Imm}{Im}

\newcommand{\simc}{\!\!\sim}

\newcommand{\Zo}{\mathbb{Z}}
\newcommand{\Ro}{\mathbb{R}}
\newcommand{\Rg}{\mathbb{R}_{\geqslant 0}}
\newcommand{\Co}{\mathbb{C}}
\newcommand{\Qo}{\mathbb{Q}}

\newcommand{\Ynon}{Y_{nsp}}
\newcommand{\Znon}{Z_{nsp}}
\newcommand{\Kspec}{K_{sp}}

\newcommand{\wh}[1]{{\widehat{#1}}}

\newcommand{\Hr}{\widetilde{H}}
\newcommand{\dd}{\partial}

\newcommand{\F}{\mathcal{F}}

\newcommand{\CP}{\mathbb{C}P}
\newcommand{\HP}{\mathbb{H}P}

\newcounter{stmcounter}[section]
\newcounter{thcounter}

\numberwithin{equation}{section}

\theoremstyle{plain}

\newtheorem{thm}[thcounter]{Theorem}

\newtheorem{prop}[stmcounter]{Proposition}
\newtheorem{lem}[stmcounter]{Lemma}

\newtheorem{clai}[stmcounter]{Claim}

\theoremstyle{definition}
\newtheorem{defin}[stmcounter]{Definition}

\theoremstyle{remark}

\newtheorem{rem}[stmcounter]{Remark}
\newtheorem{con}[stmcounter]{Construction}

\begin{document}

\title{Torus actions of complexity one in non-general position}

\author{Anton Ayzenberg}
\address{Faculty of computer science, Higher School of Economics}
\email{ayzenberga@gmail.com}

\author{Vladislav Cherepanov}
\address{Faculty of computer science, Higher School of Economics}
\email{vilamsenton@gmail.com}

\date{\today}
\thanks{This work is supported by the Russian Science Foundation under grant 18-71-00009.}



\begin{abstract}
Let the compact torus $T^{n-1}$ act on a smooth compact manifold $X^{2n}$ effectively with nonempty finite set of fixed points. We pose the question: what can be said about the orbit space $X^{2n}/T^{n-1}$ if the action is cohomologically equivariantly formal (which essentially means that $H^{\odd}(X^{2n};\Zo)=0$). It happens that homology of the orbit space can be arbitrary in degrees $3$ and higher. For any finite simplicial complex $L$ we construct an equivariantly formal manifold $X^{2n}$ such that $X^{2n}/T^{n-1}$ is homotopy equivalent to $\Sigma^3L$. The constructed manifold $X^{2n}$ is the total space of the projective line bundle over the permutohedral variety hence the action on $X^{2n}$ is Hamiltonian and cohomologically equivariantly formal. We introduce the notion of the action in $j$-general position and prove that, for any simplicial complex $M$, there exists an equivariantly formal action of complexity one in $j$-general position such that its orbit space is homotopy equivalent to~$\Sigma^{j+2}M$.
\end{abstract}

\maketitle

\section{Introduction}\label{secIntro}

Let the compact torus $T=T^k$ act effectively on a connected closed smooth manifold $X=X^{2n}$, and the action has nonempty finite set of fixed points. The number $n-k$ is called the \emph{complexity} of the action. In this paper we focus on the actions of complexity one, that is the actions of $T^{n-1}$ on $X=X^{2n}$. For an action of complexity one and a fixed point $x\in X^T$, consider $\alpha_{x,1},\ldots,\alpha_{x,n}\in \Hom(T^{n-1},T^1)\cong \Zo^{n-1}$, the weights of the tangent representation at $x$. We say that the weights are \emph{in general position} at $x\in X^T$ if any $n-1$ of $\{\alpha_{x,i}\}$ are linearly independent over $\Qo$.

The study of orbit spaces of complexity zero actions is well developed in toric topology \cite{DJ,BPnew,MasPan}. In \cite{BT}, Buchstaber and Terzic initiated the study of orbit spaces for actions of positive complexity. In particular, they proved the homeomorphisms $G_{4,2}/T^3\cong S^5$ and $F_3/T^2\cong S^4$ for the Grassmann manifold $G_{4,2}$ of complex $2$-planes in $\Co^4$ and the manifold $F_3$ of complete complex flags in $\Co^3$. It was later proved in \cite{AyzHP} that $\HP^2/T^3\cong S^5$ and $S^6/T^2\cong S^4$ for the quaternionic projective plane $\HP^2$ and the sphere $S^6$, considered as the homogeneous space $G_2/\SU(3)$.

It is easy to prove that, under certain technical assumptions \cite{AyzCompl}, the orbit space $Q^{n+1}=X^{2n}/T^{n-1}$ is a closed topological manifold if the tangent weights are in general position at each fixed point. If at least one fixed point has weights not in general position, then the orbit space $Q^{n+1}$ is a manifold with boundary~\cite{Cher}. Karshon and Tolman~\cite{KTmain} proved that the orbit space of any Hamiltonian torus action of complexity one in general position is homeomorphic to the sphere $S^{n+1}$. This includes the cases $G_{4,2}$ and $F_3$ but not~$\HP^2$ and~$S^6$.

We see that in complexity one, general position of tangent weights implies strong topological constraints on the structure of the orbit space. However, the second author studied the complexity one torus actions on regular semisimple Hessenberg varieties in \cite{Cher}: these actions are not in general position and their orbit spaces have more interesting topology. One example of the orbit space of a Hessenberg variety is the 5-sphere with four discs cut off, and another example is the 6-sphere with a tubular neighbourhood of a graph cut off. These examples made us think that the topology of orbit spaces in non-general position deserve further study and motivated our work.

In this paper, we show that dropping the assumption of general position, the orbit spaces may be almost arbitrary even if the action itself is cohomologically equivariantly formal. Recall that the action is called cohomologically equivariantly formal if its Serre spectral sequence\footnote{All coefficients are in $\Zo$ unless stated otherwise.}
\begin{equation}\label{eqSerreSeq}
E_2^{*,*}\cong H^*(BT)\otimes H^*(X)\Rightarrow H^*(X\times_TET)=H^*_T(X),
\end{equation}
collapses at $E_2$-term. It can be easily seen that the condition $H^{\odd}(X)=0$ implies equivariant formality. On the other hand, if $H^*(X)$ is $\Zo$-torsion free, then cohomological equivariant formality implies that $H^*_T(X)$ is a free $H^*(BT)$-module. If, moreover, the set of fixed points of the $T$-action is nonempty and finite, then $H^{\odd}(X)=0$, see \cite[Lm.2.1]{MasPan}. We prove the following

\begin{thm}\label{thmOrbitAnything}
For any finite simplicial complex $L$, there exists a closed smooth manifold $X_{\hat{L}}^{2n}$ with $H^{\odd}(X_{\hat{L}}^{2n})=0$, and the effective action of $T^{n-1}$ on $X_{\hat{L}}^{2n}$ with isolated fixed points and connected stabilizers such that the orbit space $Q^{n+1}=X_{\hat{L}}^{2n}/T^{n-1}$ satisfies
\begin{equation}\label{eqMainHomologyIso}
\Hr_{i+3}(Q^{n+1})\cong \Hr_i(L) \mbox{ for any } i\geqslant 0,
\end{equation}
and $\Hr_i(Q^{n+1})=0$ for $i=0,1,2$. 
\end{thm}

We mention the particular case of Theorem~\ref{thmOrbitAnything} when $L$ is the boundary $\dd\Delta^{n-1}$ of the simplex on $n$ vertices. In this case, the theorem asserts that $Q^{n+1}$ is a homology $(n+1)$-sphere. By the discussion above, there exist a large number of actions of complexity one, whose orbit space is homeomorphic to the sphere: they all correspond to weights in general position. If the weights are not in general position, the orbit space is an $(n+1)$-manifold with boundary, according to \cite{Cher}, so we cannot get $H_{n+1}(Q^{n+1})\cong\Zo$. Therefore, the case $L=\dd\Delta^{n-1}$ is exceptional in the sense that it corresponds to the general position of weights. However, this case is also covered by the proof of Theorem~\ref{thmOrbitAnything}.

The paper is organized as follows. In Section~\ref{secPrelims} we review the basic definitions and constructions needed for our arguments. Theorem~\ref{thmOrbitAnything} is proved in Section~\ref{secMainProof}: the main ingredient is the Alexander duality; we apply it twice to get the isomorphism~\eqref{eqMainHomologyIso}. Additional details about the space $X_{\hat{L}}^{2n}$, constructed in the proof of Theorem~\ref{thmOrbitAnything}, are gathered in last Section~\ref{secAlgGeom}. In Proposition~\ref{propProjBdl}, we show that $X_{\hat{L}}^{2n}$ is the total space of a projectivized line bundle over the permutohedral toric variety. Therefore, $X_{\hat{L}}^{2n}$ is a smooth projective toric variety, in particular, the torus action on $X_{\hat{L}}^{2n}$ is Hamiltonian and equivariantly formal. Next, in Proposition~\ref{propHomotopyTriple}, we prove that the orbit space $Q^{n+1}=X_{\hat{L}}^{2n}/T^{n-1}$ is actually homotopy equivalent to the triple suspension $\Sigma^3L$. Finally, in Section~\ref{secAlgGeom} we introduce the notion of a complexity one action in $j$-general position. In Theorem~\ref{thmJGeneralOrbit}, we prove that homology of the orbit space of a complexity one torus action in $j$-general position can be arbitrary in degrees $j+2$ and higher. Homology in degrees below $j+2$ vanish in our examples of $j$-general actions.

\section{Preliminaries}\label{secPrelims}

In this section, we recall the standard definitions of a locally standard torus action and a quasitoric manifold. Quasitoric manifolds were introduced in the seminal work of Davis and Januszkiewicz \cite{DJ} as a topological generalization of smooth projective toric varieties.

A smooth manifold $X^{2n}$ with an effective action of $T^n$ is called \emph{locally standard} if $X^{2n}$ has an atlas of $T^n$-invariant charts, each equivalent to an open $T^n$-invariant subset of the standard action of $T^n$ on $\Co^n\cong \Ro^{2n}$, up to some automorphism of torus. Since $\Co^n/T^n\cong \Rg^n$, the orbit space $P^n=X^{2n}/T^n$ has the natural structure of a manifold with corners. The vertices of $P^n$ correspond to the fixed points of the action.

A manifold $X^{2n}$ with a locally standard action of $T^n$ is called a \emph{quasitoric manifold} if the orbit space $P^n$ is diffeomorphic to a simple polytope as a manifold with corners. Recall that an $n$-dimensional polytope is called \emph{simple} if each of its vertices is contained in exactly $n$ facets. The same condition holds for manifolds with corners. In the following we only work with quasitoric manifolds, although some definitions below are naturally valid for more general locally standard torus actions.

Let $\{\F_1,\ldots,\F_m\}$ be the set of all facets (i.e. faces of codimension 1) of the orbit space $P$. For each facet $\F_i$ consider the subgroup $\lambda(\F_i)\subset T^n$ which stabilizes an orbit lying in the interior of $F_i$. Since the action is locally standard, $\lambda(\F_i)$ is a circle subgroup of $T^n$. Hence we may assume that $\lambda$ takes values in the lattice $\Zo^n\cong \Hom(T^1,T^n)$ of 1-dimensional subgroups of $T^n$:
\begin{equation}\label{eqCharFuncDef}
\lambda\colon \{\F_1,\ldots,\F_m\}\to \Hom(T^1,T^n)\cong\Zo^n.
\end{equation}
It should be noticed that the value of $\lambda$ is determined up to sign unless some omniorientation is imposed on $X^{2n}$, see details in \cite{BPnew}. The function $\lambda$ is called a \emph{characteristic function} of the manifold $X^{2n}$. The condition of a locally standard action implies that whenever distinct facets $\F_{i_1},\ldots,\F_{i_n}$ intersect at a vertex, the values $\lambda(\F_{i_1}),\ldots,\lambda(\F_{i_n})$ form the basis of the lattice $\Zo^n$. This condition is called the \emph{$(*)$-condition}. Hence, with any quasitoric manifold $X^{2n}$, one can associate the characteristic pair $(P^n,\lambda)$ consisting of the simple polytope $P^n$ and the characteristic function \eqref{eqCharFuncDef}.

It is possible to reconstruct any quasitoric manifold $X^{2n}$ from its characteristic pair $(P^n,\lambda)$. Given a simple polytope $P^n$ and a function \eqref{eqCharFuncDef} satisfying the $(*)$-condition, consider the reduced space
\begin{equation}\label{eqReducedSpaceQToric}
X^{2n}_{(P,\lambda)}=(P^n\times T^n)/\simc
\end{equation}
where the identification $\sim$ is generated by the equivalences of the form $(x_1,t_1)\sim(x_2,t_2)$ if $x_1=x_2\in \F_i$ and $t_1t_2^{-1}\in \lambda(\F_i)$. Then $X^{2n}_{(P,\lambda)}$ is a topological manifold carrying the action of $T^n$; it is equivariantly homeomorphic to the original manifold $X^{2n}$, see \cite{DJ}. A different approach was developed in \cite{BPR} to construct a smooth model $X^{2n}_{(P,\lambda)}$ of a quasitoric manifold.

Quasitoric manifolds $X^{2n}$ provide examples of toric actions of complexity zero. Davis and Januszkiewicz \cite{DJ} proved that $H^{\odd}(X^{2n})=0$. This means that quasitoric manifolds are cohomologically equivariantly formal. On the other hand, their orbit spaces are polytopes by definition. This setting was further extended by Masuda and Panov in \cite{MasPan}: they gave a criterion, in terms of the orbit space, for the equivariant formality of a complexity zero torus action. In particular, it follows from their work that, for any equivariantly formal torus action of complexity zero with nonempty finite set of fixed points, the orbit space is a homology disc.

Our current work was motivated by a similar question for the actions of complexity one, that are the actions of $T^{n-1}$ on $2n$-manifolds. Is it possible to characterize equivariant formality in terms of the topology of the orbit space? The results of \cite{Cher} suggested that the answer is negative: the complexity one torus actions on regular semisimple Hessenberg varieties are equivariantly formal, but they have orbit spaces with nontrivial topology.

A natural way to construct many complexity one actions is to take a complexity zero action of $T^n$ on a quasitoric manifold $X^{2n}$ and consider the induced action of some subtorus $T^{n-1}\subset T^n$. Examples of such actions were considered in \cite{AyzCompl}. It was proved that whenever the induced action of $T^{n-1}$ is in general position (recall the definition in Section~\ref{secIntro}) then the orbit space $X^{2n}/T^{n-1}$ is homeomorphic to a sphere. In the current paper, we concentrate on the situation when the induced action of $T^{n-1}$ on a quasitoric manifold $X^{2n}$ is not in general position. However, our arguments include the case of general position and allow to recover the result of~\cite{AyzCompl}.

%

We now introduce some general notation to work with the fixed points of locally standard actions. Let $X^{2n}$ be a quasitoric manifold and $x\in X^{2n}$ be a fixed point. Let $N$ denote the lattice of 1-dimensional subgroups, $N=\Hom(T^1,T^n)$. Let $\lambda_1,\ldots,\lambda_n\in \Hom(T^1,T^n)=N$ be the characteristic vectors of the $T^n$-action at $x$, and $\wh{\alpha}_1,\ldots,\wh{\alpha}_n\in \Hom(T^n,T^1)=N^*$ be the weights of its tangent representation. It is assumed that the $n$ facets of $P^n=X^{2n}/T^n$ adjacent to $x$ are enumerated from $1$ to $n$ and $\lambda_i$ is the value of characteristic function at $i$-th facet. Given some enumeration of facets around $x$, we can enumerate the edges of $P^n$ adjacent to $x$ in a canonical way: the $i$-th edge is the only one which is not contained in the $i$-th facet. Then $\wh{\alpha}_i$ is the weight corresponding to $i$-th edge. With this convention, there holds $\langle \lambda_i,\wh{\alpha}_j\rangle = \delta_{ij}$. 

Given a subtorus $T^{n-1}\subset T^n$, we get a short exact sequence of tori
\[
T^{n-1}\stackrel{i}{\longrightarrow} T^n\stackrel{p}{\longrightarrow} T^n/T^{n-1}\cong T^1.
\]
It induces the exact sequences of lattices
\[
\Hom(T^1,T^{n-1})\stackrel{i_*}{\longrightarrow} \Hom(T^1,T^n)\stackrel{p_*}{\longrightarrow} \Hom(T^1,T^n/T^{n-1})\cong \Zo,
\]
\[
\Zo\cong \Hom(T^n/T^{n-1},T^1)\stackrel{p^*}{\longrightarrow}\Hom(T^n,T^1)\stackrel{i^*}{\longrightarrow} \Hom(T^{n-1},T^1).
\]
The hyperplane sublattice $\Ker p_*\subset N=\Hom(T^1,T^n)$ will play an important role in our arguments. We denote it by $\Pi$:
\begin{equation}\label{eqDefHyperplanePi}
\Pi=\Ker (p_*\colon N\to \Zo)=\Imm(i_*\colon \Hom(T^1,T^{n-1})\to N).
\end{equation}

\begin{prop}[\cite{AyzCompl}]
The restricted action of $T^{n-1}$ on a quasitoric manifold $X^{2n}$ is in general position at a fixed point $x$ if and only if $p_*(\lambda_j)\neq 0$ for any characteristic value $\lambda_j$ at $x$ (which means that all characteristic vectors $\lambda_j$ do not lie in $\Pi$). The restricted action is in general position globally if and only if all characteristic values $\lambda(\F_j)$ do not lie in $\Pi$.
\end{prop}

Certainly, if the restricted action is in general position at a fixed point $x$, this implies that $x$ is isolated with respect to $T^{n-1}$ (otherwise one of the weights would be zero which contradicts the linear independence). However, since we are going to work with actions not in general position, we need a convenient criterion to check that the restricted action still has isolated fixed points. It is given by the following lemma.

\begin{lem}\label{lemNonIsolatedVertex}
Let $x$ be a fixed point of the $T^n$-action on a quasitoric manifold $X^{2n}$. Then $x$ is not isolated for the restricted action of $T^{n-1}\subset T^n$ if and only if some $n-1$ of its characteristic values lie in $\Pi$.
\end{lem}

\begin{proof}
The point $x$ is not isolated for the restricted action if and only if at least one of its weights, say $\wh{\alpha}_j$, is mapped to zero by $i^*$. Hence $\wh{\alpha}_j$ belongs to the image of $p^*$, which means that the elements $\lambda_1,\ldots,\widehat{\lambda_j},\ldots,\lambda_n$ of the dual basis are annihilated by $p_*$. This proves the statement.
\end{proof}

We also need a condition that guarantees that the restricted action of $T^{n-1}$ on $X^{2n}$ has connected stabilizers.

\begin{lem}\label{lemConnectedStabilizers}
The induced action of $T^{n-1}$ on $X$ has connected stabilizers if and only if, for any collection of intersecting facets $\F_1,\ldots,\F_k$, the quotient abelian group
\[
\Pi/(\langle\lambda(\F_1),\ldots,\lambda(\F_k)\rangle\cap\Pi)
\]
is torsion-free.
\end{lem}

\begin{proof}
Let $x\in P$, and assume that $x$ lies in the relative interior of the face $F$ given by the intersection of facets $\F_1,\ldots,\F_k$. Then the $T^n$-stabilizer of $x$ is the product of 1-dimensional subgroups
\[
\lambda(\F_1)(T^1)\times\cdots\times\lambda(\F_k)(T^1)\subseteq T^n.
\]
The $T^{n-1}$-stabilizer of $x$ is the intersection $\lambda(\F_1)(T^1)\times\cdots\times\lambda(\F_k)(T^1)\cap T^{n-1}$. The fact that this stabilizer is connected corresponds, on the level of lattices, to the freeness of the group $\Pi/(\langle\lambda(\F_1),\ldots,\lambda(\F_k)\rangle\cap\Pi)$.
\end{proof}

\begin{defin}
Let $X^{2n}$ be the quasitoric manifold determined by a characteristic pair $(P^n,\lambda)$. For the induced action of a subtorus $T^{n-1}\subset T^n$ on $X^{2n}$, consider the hyperplane $\Pi=\Ker p_*$ defined by \eqref{eqDefHyperplanePi}. A facet $\F_j$ of $P^n$ is called \emph{special} if $\lambda(\F_j)\in \Pi$. Similarly, a proper face $F$ of $P^n$ is called \emph{special} if all facets containing $F$ are special. All other proper faces of $P^n$ are called non-special. A point $x\in P^n$ is called \emph{special} if it lies in a relative interior of a special face. A point is non-special if it lies in a (closed) non-special face.
\end{defin}

The definition implies that the union
\[
\Ynon=\bigcup_{F \mbox{ non-special}}F\subset \dd P^n,
\]
which is the subset of all non-special points of $\dd P^n$, is a closed subset of the boundary $\dd P^n$. Moreover, if $P^n$ is considered as a CW-complex, with the cell structure given by its faces, then $\Ynon$ is its CW-subcomplex.

As before, we denote the orbit space $X^{2n}/T^{n-1}$ by $Q^{n+1}$. We have the residual action of $T^n/T^{n-1}$ on $Q^{n+1}$, and the orbit space of this action is a simple polytope $P^n\cong X^{2n}/T^{n}$. Let $p\colon Q^{n+1}\to P^n$ denote the projection map of the residual action. If $x\in P^n$, then the preimage $p^{-1}(x)$ can be either a circle or a point.

\begin{lem}\label{lemSpecialPoints}
The preimage $p^{-1}(x)$ is a circle if and only if $x$ is special.
\end{lem}

The proof is a direct check.

\begin{lem}\label{lemCollapsedNonspecial}
Assume that the induced action of $T^{n-1}\subset T^n$ on a quasitoric manifold $X^{2n}$ has connected stabilizers. Then the orbit space $Q^{n+1}=X^{2n}/T^{n-1}$ is obtained from the product $P^n\times S^1$ by collapsing circles over non-special points:
\begin{equation}\label{eqQuotientRepresentation}
Q^{n+1}\cong(P^n\times S^1)/\simc_{\Ynon},
\end{equation}
where $(x,s_1)\sim_{\Ynon}(x,s_2)$ if $x\in \Ynon$.
\end{lem}

\begin{proof}
The statement easily follows from the definition of a quasitoric manifold as the reduced space \eqref{eqReducedSpaceQToric}. Indeed, by collapsing the subtorus $T^{n-1}$ in each fiber of \eqref{eqReducedSpaceQToric}, we get \eqref{eqQuotientRepresentation}.
\end{proof}

\section{Proof of Theorem~\ref{thmOrbitAnything}}\label{secMainProof}

\begin{con}
Assume that the induced action of $T^{n-1}\subset T^n$ on a quasitoric manifold $X^{2n}$ has connected stabilizers. Consider the product $P^n\times D^2$, and the following subcomplex of its boundary
\begin{equation}\label{eqDefZnon}
\Znon=(\Ynon\times D^2)\cup (P^n\times \dd D^2)\subset (\dd P^n\times D^2)\cup (P^n\times \dd D^2)=\dd(P^n\times D^2)\cong S^{n+1},
\end{equation}
where $\Ynon\subset \dd P^n$ is the union of nonspecial faces introduced in Section~\ref{secPrelims}.
\end{con}

\begin{lem}\label{lemBlowNonSpecial}
The complex $\Znon$ is homotopy equivalent to the quotient $Q^{n+1}=X^{2n}/T^{n-1}$.
\end{lem}

\begin{proof}
There is a natural map from $\Znon$ to the reduced space $(P^n\times S^1)/\simc_{\Ynon}$, which collapses a 2-disc over any nonspecial point of $\dd P^n$ to a point. Both spaces $\Znon$ and $(P^n\times S^1)/\simc_{\Ynon}$ are compact CW-complexes, the map is proper cellular. All fibers of this map are contractible, hence the map is a homotopy equivalence according to the result of Smale \cite{Smale}. Now, by Lemma~\ref{lemCollapsedNonspecial}, $Q^{n+1}\cong (P^n\times S^1)/\simc_{\Ynon}$.
\end{proof}

\begin{con}
Let $K_P$ be the nerve-complex of a simple polytope $P=P^n$. This means $K_P$ is the boundary of a simplicial polytope $P^*$ dual to $P$, or equivalently, $K_P$ has vertex set $[m]$, and $\{i_1,\ldots,i_s\}\in K_P$ if and only if the corresponding facets $\F_{i_1},\ldots,\F_{i_s}\subset P$ intersect. Given an action of a subtorus $T^{n-1}$ on a quasitoric manifold $X^{2n}$ as before, we introduce the subcomplex $\Kspec\subset K_P$, such that $\{i_1,\ldots,i_s\}\in \Kspec$ if and only if the face $\F_{i_1}\cap\cdots\cap\F_{i_s}$ is a special face of $P$. According to the definition of special faces, this condition simply means that all facets $\F_{i_1},\ldots,\F_{i_s}$ are special. Henceforth $\Kspec$ is a full subcomplex on the vertex set $\{i\in[m]\mid \F_i\mbox{ is special}\}$.
\end{con}

\begin{lem}\label{lemMainAlexanderDual}
The spaces $\Znon$ and $\Kspec$ are Alexander dual in the sphere $S^{n+1}$. There holds $\Hr_i(Q^{n+1})\cong \Hr^{n-i}(\Kspec)$.
\end{lem}

\begin{proof}
Recall that $\Znon$ is a subset of the sphere $\dd(P^n\times D^2)\cong S^{n+1}$. For the complement $S^{n+1}\setminus \Znon$ we have
\[
\dd(P^n\times D^2)\setminus\Znon=\bigcup_{F\mbox{ special}}(\relint F\times \relint D^2).
\]
The union on the right-hand side is homotopy equivalent to its nerve which is the simplicial complex $\Kspec$. The second statement follows from the Alexander duality $\Hr_i(\Znon)\cong \Hr^{n-i}(\Kspec)$ and the homotopy equivalence $\Znon\simeq Q^{n+1}$ given by Lemma~\ref{lemBlowNonSpecial}.
\end{proof}

\begin{lem}\label{lemVanishingOfH012}
For a restricted action of $T^{n-1}$ on a quasitoric manifold $X^{2n}$, having isolated fixed points, there holds $\Hr_i(Q^{n+1})=0$ for $i=0,1,2$.
\end{lem}

\begin{proof}
The simplicial complex $\Kspec$ has dimension at most $n-3$. Indeed, otherwise there would exist $n-1$ intersecting facets $\F_{i_1},\ldots,\F_{i_{n-1}}$ whose characteristic values lie in $\Pi$. In this case the action of $T^{n-1}$ would have non-isolated fixed points by Lemma~\ref{lemNonIsolatedVertex}. This contradicts the assumption. 

Now, since $\dim \Kspec\leqslant n-3$, we have $\Hr^{i}(\Kspec)=0$ for $i\geqslant n-2$. The statement follows from the Alexander duality $\Hr_i(Q^{n+1})\cong \Hr^{n-i}(\Kspec)$ given by Lemma~\ref{lemMainAlexanderDual}.
\end{proof}

\begin{rem}
In \cite{Cher}, an argument similar to Lemma~\ref{lemVanishingOfH012} was applied to show that homology in degrees 0,1,2 vanish for the orbit spaces of complexity one torus actions on regular semisimple Hessenberg varieties. We suppose that vanishing of homology in degrees 0,1,2 is a general phenomenon for the orbit spaces of equivariantly formal torus actions of complexity one with isolated fixed points.
\end{rem}

\begin{con}\label{conCombAlex}
We recall the constructions of the barycentric subdivisions and the combinatorial Alexander duality. The reader is referred to \cite[Section~2.4]{BPnew} for details. Let $L$ be an abstract simplicial complex on a finite vertex set $[n]=\{1,2,\ldots,n\}$. Then the \emph{barycentric subdivision} $L'$ is the simplicial complex on the vertex set $L\setminus\{\varnothing\}$ such that $\{I_1,\ldots,I_s\}\in L'$ if and only if the simplicies $I_1,\ldots,I_s\in L$ form a nested family. This means $I_1\subset I_2\subset\cdots\subset I_s$, probably after some permutation of indices. Geometrical realizations of $L$ and $L'$ are homeomorphic.

Let $L$ be a simplicial complex on a vertex set $[n]$ such that $L$ is not the whole simplex on $n$ vertices. The \emph{combinatorial Alexander dual complex} $\hat{L}$ is defined on the same vertex set $[n]$ as follows
\[
\hat{L}=\{I\subset[n]\mid [n]\setminus I\notin L\}.
\]
Note that if $\dim L=n-2$, then $\hat{L}$ has ghost vertices (a ghost vertex of $\hat{L}$ is a vertex $i\in[n]$ such that $\{i\}\notin \hat{L}$). In the following arguments we allow the complex $\hat{L}$ to have ghost vertices. However, it will be assumed that $L$ does not have ghost vertices.

Any simplicial complex $L\neq\Delta^{n-1}$ on the set $[n]$ can be considered as a subcomplex of $\dd\Delta^{n-1}$. Applying barycentric subdivisions, we get the embedding $L'\subseteq (\dd\Delta^{n-1})'$. Similarly, there is an embedding of $\hat{L}'$ into $(\dd\Delta^{n-1})'$ which sends the vertex $I\in \hat{L}\setminus \varnothing$ of $\hat{L}'$ to the vertex $[n]\setminus I$ of $(\dd\Delta^{n-1})'$. The subcomplexes $L'$ and $\hat{L}'$ of the $(n-2)$-dimensional sphere $(\dd\Delta^{n-1})'$ are Alexander dual. This implies the combinatorial Alexander duality
\[
\Hr_i(L)\cong \Hr^{n-3-i}(\hat{L}).
\]
\end{con}

\begin{con}\label{conPermut}
We recall the definition of a permutohedron and a permutohedral toric variety, referring to \cite{Postn} for missing details. The \emph{permutohedron} $\Pe^{n-1}$ is the convex polytope
\[
\Pe^{n-1}=\conv\{(\sigma(b_1),\ldots,\sigma(b_n))\mid\sigma\in\Sigma_n\}
\]
where $b_1<b_2<\cdots<b_n$. The combinatorial type of $\Pe^{n-1}$ does not depend on a choice of $b_i$. The permutohedron is determined by the following affine inequalities \cite{Rado}: a point $x=(x_1,\ldots,x_n)\in\Ro^n$ lies in $\Pe^{n-1}$ if and only if
\[
x_1+\cdots+x_n=b_1+\cdots+b_n
\]
and
\begin{gather*}
  x_i\leqslant b_n\quad \forall i\in[n];\\
  x_i+x_j\leqslant b_{n-1}+b_{n}\quad \forall \{i,j\}\subset [n];\\
  x_i+x_j+x_k\leqslant b_{n-2}+b_{n-1}+b_{n}\quad \forall \{i,j,k\}\subset [n];\\
  \cdots\\
  \sum\nolimits_{i\in S} x_i\leqslant b_2+b_3+\cdots+b_n\quad \forall S\subset [n], |S|=n-1.
\end{gather*}
There are no redundant inequalities in this list. This means that the facets of $\Pe^{n-1}$ are indexed by the subsets $S\subset [n]$, $S\neq [n], \varnothing$: the facet $\F_S$ corresponding to a subset $S\subset [n]$ is given by the equation $\sum_{i\in S}x_i=\sum_{i=n-|S|+1}^nb_i$. The polytope $\Pe^{n-1}$ is simple. The facets $\F_{S_1},\ldots,\F_{S_k}$ intersect if and only if their indexing sets $S_1,\ldots,S_k$ form a nested collection (that is $S_1\subset\cdots\subset S_k$, probably, after some permutation). Therefore, the simplicial sphere dual to $\Pe^{n-1}$ is the barycentric subdivision of the boundary of the $(n-1)$-simplex:
\[
\dd(\Pe^{n-1})^*=(\dd\Delta^{n-1})'
\]
The vertex of $(\dd\Delta^{n-1})'$ corresponding to the facet $\F_S$ will be denoted by $i_S$.

Let $e_1,\ldots,e_{n-1}$ be the outward unit normal vectors to the facets $\F_{\{1\}},\ldots,\F_{\{n-1\}}$ of $\Pe^{n-1}$ inside the affine hyperplane $\{x_1+\cdots+x_n=b_1+\cdots+b_n\}$. Then $e_n=-\sum_{i=1}^{n-1}e_i$ is the outward unit normal vector to the facet $\F_{\{n\}}$. Let $N^{n-1}\cong \Zo^{n-1}$ be the lattice generated by $e_1,\ldots,e_{n-1}$. For an arbitrary proper subset $S\subset[n]$ the outward normal vector $\nu_S$ to the facet $\F_S$ has the form
\[
\nu_S=\sum_{i\in S}e_i.
\]
The normal fan of $\Pe^{n-1}$ is nonsingular, its maxinal cones are formed by Weyl chambers of type A. The normal fan of the permutohedron $\Pe^{n-1}$ hence defines a nonsingular projective toric variety $X_{\Pe}^{2n-2}$, called the \emph{permutohedral variety}. This variety is well known and found many applications in algebraic geometry, algebraic combinatorics, and representation theory (see, e.g. \cite{Postn,Huh}, and references therein). As a quasitoric manifold, $X_{\Pe}^{2n-2}$ is defined by the characteristic pair $(\Pe^{n-1},\nu)$ where $\nu(F_S)=\nu_S=\sum_{i\in S}e_i$.
\end{con}

With all the preparatory work done, we now prove Theorem~\ref{thmOrbitAnything}.

\begin{proof}[Proof of Theorem~\ref{thmOrbitAnything}]
Consider a finite simplicial complex $L$ on the vertex set $[n]$. Without loss of generality assume $L\neq\Delta^{n-1}$ (otherwise, $\Delta^{n-1}$ can be replaced by any other acyclic complex in order to prove the statement). Therefore, the combinatorial Alexander dual complex $\hat{L}=\{[n]\setminus I\mid I\notin L\}$ is well defined. It is also assumed that $L$ does not have ghost vertices. This assumption implies
\begin{equation}\label{eqDimLhat}
\dim\hat{L}\leqslant n-3,
\end{equation}
by Construction \ref{conCombAlex}. We have $\Hr_i(L)\cong \Hr^{n-3-i}(\hat{L})$ by combinatorial Alexander duality.

The idea of the proof is the following: we construct a polytope of dimension~$n$, a quasitoric manifold over this polytope, and a subtorus $T^{n-1}\subset T^n$, such that the complex $\Kspec$ for these data coincides with $\hat{L}$. Then Alexander duality between $\hat{L}$ and $L$ in $(\dd\Delta^{n-1})'\cong S^{n-2}$ and Alexander duality between $\Kspec=\hat{L}$ and $\Znon\simeq Q^{n+1}$, given by Lemma~\ref{lemMainAlexanderDual}, would imply
\begin{equation}\label{eqTwoAlexanderDualities}
\Hr_i(L)\cong \Hr^{n-3-i}(\hat{L})=\Hr^{n-3-i}(\Kspec)\cong \Hr_{i+3}(Q^{n+1}).
\end{equation}

For a polytope, we take the prism with a permutohedron in the base, that is $P^n=\Pe^{n-1}\times I^1$. Its dual simplicial sphere is the suspended barycentric subdivision of the boundary of the $(n-1)$-simplex:
\[
K_P=\Sigma(\dd\Delta^{n-1})'.
\] 
The facets of $P^n$, corresponding to vertices of $K_P$, are the side facets $\F_S\times I^1$ of the prism, defined for any proper subset $S\subset[n]$, and the two bases $\F_a$ and $\F_b$ of the prism.

Recall that the normal vectors of the permutohedron $\Pe^{n-1}$ lie in the lattice $N^{n-1}=\langle e_1,\ldots,e_{n-1}\rangle$, and we adopt the convention $e_n=-\sum_{i=1}^{n-1}e_i$. The characteristic function $\lambda_{\hat{L}}\colon P^n\to N^{n-1}\times\Zo\cong\Zo^n$ is defined as follows. We set $\lambda_{\hat{L}}(\F_a)=(0,1)\in N\times \Zo$, $\lambda_{\hat{L}}(\F_b)=(0,-1)\in N\times \Zo$, and
\begin{equation}\label{eqLambdaDef}
\lambda_{\hat{L}}(\F_S\times I^1)=\begin{cases}
                          (\nu_S,0), & \mbox{if } v_S\in \hat{L}' \\
                          (\nu_S,1), & \mbox{if }v_S\notin\hat{L}',
                        \end{cases}
\end{equation}
where $\nu_S$ are the normal vectors to the facets of the permutohedron.

\begin{clai}
The function $\lambda_{\hat{L}}$ is a characteristic function on $P^n=\Pe^{n-1}\times I^1$.
\end{clai}

\begin{proof}
Any vertex of $P^n$ is the intersection of facets
\[
\F_{S_1}\times I^1,\F_{S_2}\times I^1,\ldots,\F_{S_{n-1}}\times I^1, F_a\mbox{ or } F_b,
\]
for some nested sequence $S_1,\ldots,S_{n-1}\subset[n]$. The characteristic values at these facets are
\[
(\nu_{S_1},\kappa_1), (\nu_{S_2},\kappa_2),\quad\ldots\quad,(\nu_{S_{n-1}},\kappa_{n-1}), (0,\pm1),
\]
where each $\kappa_i$ is either $0$ or $1$. Subtracting or adding the last vector to those $\lambda(\F_{S_i}\times I^1)$ with $\kappa_i=1$, we get the vectors
\[
\nu_{S_1},\nu_{S_2},\ldots\nu_{S_{n-1}},\pm\epsilon,
\]
which obviously form a basis of the lattice $N^{n-1}\times\Zo$.
\end{proof}

Let $X_{\hat{L}}^{2n}$ be the quasitoric manifold determined by the characteristic pair $(P,\lambda_{\hat{L}})$. Consider the subtorus $T^{n-1}\subset T^n$ which corresponds to the first $n-1$ coordinates of the torus. This means that the kernel hyperplane $\Pi$ defined by \eqref{eqDefHyperplanePi} is the coordinate hyperplane
\[
\Pi=N^{n-1}\times\{0\}\subset N^{n-1}\times\Zo.
\]

\begin{clai}
The induced action of $T^{n-1}$ on the quasitoric manifold $X_{\hat{L}}^{2n}$ has isolated fixed points.
\end{clai}

\begin{proof}
According to Lemma~\ref{lemNonIsolatedVertex}, we should check that there do not exist $n-1$ intersecting facets $\F_1,\ldots,\F_{n-1}$ of $P$ such that $\lambda(\F_1),\ldots,\lambda(\F_{n-1})\in\Pi$. Assume that there exists such $(n-1)$-tuple. By construction, the only facets $\F$ with $\lambda(\F)\in\Pi$ are the side facets whose indices lie in $\hat{L}'$. Then we get $\dim\hat{L}'\geqslant n-2$ which contradicts \eqref{eqDimLhat}.
\end{proof}

\begin{clai}
The induced action of $T^{n-1}$ on the quasitoric manifold $X_{\hat{L}}^{2n}$ has connected stabilizers.
\end{clai}

\begin{proof}
A direct check using Lemma~\ref{lemConnectedStabilizers} and the possibility to make the second coordinate of all characteristic vectors zero by subtracting $\lambda_{\hat{L}}(\F_a)=(0,1)$ or adding $\lambda_{\hat{L}}(\F_b)=(0,-1)$.
\end{proof}

\begin{clai}
For the induced action of $T^{n-1}$ on the quasitoric manifold $X_{\hat{L}}^{2n}$, we have $\Kspec=\hat{L}'\subset \Sigma(\dd\Delta^{n-1})'$.
\end{clai}

\begin{proof}
By definition, the vertices of $\Kspec$ are all indices $i\in\Sigma(\dd\Delta^{n-1})'$ such that $\lambda(\F_i)\in\Pi=N^{n-1}$. These are exactly the vertices of $\hat{L}'\subset (\dd\Delta^{n-1})'\subset \Sigma(\dd\Delta^{n-1})'$ according to~\eqref{eqLambdaDef}. Now, both $\Kspec$ and $\hat{L}'$ are the full subcomplexes on their vertex sets hence they coincide.
\end{proof}

The part of Theorem~\ref{thmOrbitAnything} concerning $\Hr_i(Q^{n+1})$ for $i\geqslant 3$ now follows from \eqref{eqTwoAlexanderDualities} applied to the constructed manifold~$X_{\hat{L}}^{2n}$. The vanishing of $\Hr^i(Q^{n+1})$ for $i=0,1,2$ was proved in Lemma~\ref{lemVanishingOfH012} for general restricted actions.
\end{proof}

\section{Details and generalizations}\label{secAlgGeom}

\begin{prop}\label{propProjBdl}
The quasitoric manifold $X_{\hat{L}}^{2n}$ constructed in Section~\ref{secMainProof} is a smooth projective toric variety. It is the total space of the projective line bundle over the permutohedral toric variety $X^{2n-2}_{\Pe}$.
\end{prop}

\begin{proof}
Indeed, we start with a permutohedral variety $X_{\Pe}^{2n-2}$ and consider the line bundle $\xi_{\hat{L}}$ whose first Chern class, or Cartier divisor, is supported on the subcomplex $\hat{L}'$, that is
\[
c_1(\xi_{\hat{L}'})=\sum_{i\in \hat{L}'\subset(\dd\Delta^{n-1})'}v_i\in H^2(X^{2n-2}_{\Pe}),
\]
where $v_i$ are the standard generators, or divisors, corresponding to the facets of $\Pe^{n-1}$. Consider the $\CP^1$-bundle over $X^{2n-2}_{\Pe}$ given by the projectivisation
\begin{equation}\label{eqProjLineBdl}
\mathbb{P}(\xi_{\hat{L}'}\oplus\underline{\Co}),
\end{equation}
where $\underline{\Co}$ denotes the trivial line bundle over $X^{2n-2}_{\Pe}$. The total space of the projective line bundle \eqref{eqProjLineBdl} is a smooth projective toric variety (see \cite[Sect.7.3]{CLH}). Fans of toric varieties given by projectivizations  of sums of line bundles are described in detail in \cite[Prop.7.3.3]{CLH}. In our case, this construction gives the following: the moment polytope of $\mathbb{P}(\xi_{\hat{L}'}\oplus\underline{\Co})$ combinatorially coincides with $\Pe^{n-1}\times I^1$, and the 1-dimensional cones of a fan are generated by the vectors $\lambda_{\hat{L}}(\F_i)$, defined by \eqref{eqLambdaDef}. Hence $\mathbb{P}(\xi_{\hat{L}'}\oplus\underline{\Co})$ is equivariantly homeomorphic to $X_{\hat{L}}^{2n}$.
\end{proof}

In the main part of the paper, we have concentrated on homological properties of manifolds with torus actions and their orbit spaces. However, it should be mentioned that Theorem~\ref{thmOrbitAnything} can be strengthened as follows.

\begin{prop}\label{propHomotopyTriple}
Let $L$ be a finite simplicial complex and the quasitoric manifold $X_{\hat{L}}^{2n}$ with the action of $T^{n-1}$ be as in Section~\ref{secMainProof}. Then $X_{\hat{L}}^{2n}/T^{n-1}$ is homotopy equivalent to~$\Sigma^3L$.
\end{prop}

\begin{proof}
Recall that the underlying simplicial sphere of a quasitoric manifold $X_{\hat{L}}^{2n}$ has the form $\Sigma(\dd\Delta^{n-1})'$, and the special subcomplex $\Kspec=\hat{L}$ is a subcomplex of the equatorial sphere $(\dd\Delta^{n-1})'\cong S^{n-2}$. Therefore the Alexander dual of $\Kspec$ inside $\Sigma(\dd\Delta^{n-1})'\cong S^{n-1}$ is homotopy equivalent to $\Sigma L$. On the other hand, the space $\Ynon$ is homotopy equivalent to the complement of $\Kspec$ in $S^{n-1}=\dd P^n\cong\Sigma(\dd\Delta^{n-1})'$. Therefore $\Ynon\simeq\Sigma L$.

It can be seen from~\eqref{eqDefZnon}, that the subspace $\Znon$ is homotopy equivalent to $\Sigma^2\Ynon$. Hence $Q^{n+1}\simeq \Znon\simeq\Sigma^2\Ynon\simeq\Sigma^3L$.
\end{proof}

Now we prove a refinement of Theorem~\ref{thmOrbitAnything} and some other results for a more specific class of torus actions of complexity one.

\begin{defin}\label{definJGeneral}
Assume that an action of $T^{n-1}$ on a manifold $X^{2n}$ is effective and has nonempty finite set of fixed points. We say that this action is in \emph{$j$-general position} if, at each fixed point $x$, every $j$ of the tangent weights $\{\alpha_{x,1},\ldots,\alpha_{x,n}\in \Hom(T^{n-1},T^1)\}$ are linearly independent.
\end{defin}

Every action with isolated fixed points is in $1$-general position. A ``general position'' is a synonym for an ``$(n-1)$-general position''. Then we have the following refinement of Theorem~\ref{thmOrbitAnything}

\begin{thm}\label{thmJGeneralOrbit}
For any finite simplicial complex $M$, there exists a closed smooth manifold $X^{2n}$ with $H^{\odd}(X^{2n})=0$, and an action of $T^{n-1}$ on $X^{2n}$ in $j$-general position such that the orbit space $Q^{n+1}=X^{2n}/T^{n-1}$ is homotopy equivalent to $\Sigma^{j+2}M$.
\end{thm}

The construction and the arguments remain essentially the same as in Theorem~\ref{thmOrbitAnything} and Proposition~\ref{propHomotopyTriple}. Let us prove several lemmas generalizing the previous arguments. 

\begin{lem}\label{lemJGeneralForRestrictionLoc}
Let $x$ be a fixed point of the $T^n$-action on a quasitoric manifold $X^{2n}$. Then $x$ is not in $j$-general position for the restricted action of $T^{n-1}\subset T^n$ if and only if some $n-j$ of its characteristic values at $x$ lie in $\Pi$.
\end{lem}

This Lemma and its proof is completely similar to Lemma~\ref{lemNonIsolatedVertex}. Lemma~\ref{lemJGeneralForRestrictionLoc} implies the generalization of~\eqref{eqDimLhat}.

\begin{lem}\label{lemJgeneralForRestriction}
The action of $T^{n-1}\subset T^n$ on a quasitoric manifold $X^{2n}$ is in $j$-general position if and only if $\dim \Kspec\leqslant n-2-j$.
\end{lem}

Next we have the extension of Lemma~\ref{lemVanishingOfH012}

\begin{lem}
For a restricted action of $T^{n-1}$ on a quasitoric manifold $X^{2n}$ in $j$-general position, there holds $\Hr_i(Q^{n+1})=0$ for $i=0,1,\ldots,j+1$.
\end{lem}

\begin{proof}
The Alexander duality given by Lemma~\ref{lemMainAlexanderDual} implies $\Hr_i(Q^{n+1})\cong \Hr^{n-i}(\Kspec)$. The homology groups $\Hr^{n-i}(\Kspec)$ vanish by dimensional reasons for $i=0,1,\ldots,j+1$ according to Lemma~\ref{lemJgeneralForRestriction}. This implies the statement.
\end{proof}

Recall that a simplicial complex $L$ is called \emph{$j$-neighborly} if any $j$ of its vertices form a simplex (that is $L$ contains the $(j-1)$-skeleton of a simplex). The condition $1$-neighborly simply means that $L$ does not have ghost vertices. This condition is assumed to hold for~$L$.

\begin{proof}[Proof of Theorem~\ref{thmJGeneralOrbit}]
Assume that $L$ has $n$ vertices. It is easily seen from the definition of combinatorial Alexander duality that $\dim \hat{L}\leqslant n-2-j$ if and only if $L$ is $j$-neighborly. So far, Theorem~\ref{thmJGeneralOrbit} can be deduced from Proposition~\ref{propHomotopyTriple} and Lemma~\ref{lemJgeneralForRestriction} if we take $X^{2n}=X_{\hat{L}}^{2n}$ for a simplicial complex $L$ satisfying two properties:
\begin{enumerate}
  \item $L$ is $j$-neighborly;
  \item $L$ is homotopy equivalent to the $(j-1)$-fold suspension $\Sigma^{j-1}M$.
\end{enumerate}
To prove the existence of a simplicial complex with these properties, we introduce an operation $s(\cdot)$ on simplicial complexes which homotopically acts as the suspension but raises the degree of neighborliness. Let $K$ be a simplicial complex on the vertex set $V$, $|V|=m$. Consider the simplicial complex $s(K)$ on $m+1$ vertices given by
\begin{equation}\label{eqSKoperation}
s(K)=\Cone K\cup \Delta_V^{m-1}
\end{equation}
(we take the cone over $K$ and add a simplex on the whole set $V$). Then
\begin{enumerate}
  \item If $K$ is $r$-neighborly, then $s(K)$ is $(r+1)$-neighborly. 
  \item $s(K)$ is homotopy equivalent to the suspension $\Sigma K$. Indeed, by collapsing the simplex $\Delta_V^{m-1}$ in~\eqref{eqSKoperation} to a point, we get the suspension $\Sigma K$.
\end{enumerate}
The complex $L$ can be obtained by applying the operation $s(\cdot)$ to $M$ $j-1$ times: $L=s^{j-1}(M)$.
Theorem~\ref{thmJGeneralOrbit} follows by taking $X^{2n}=X_{\hat{L}}^{2n}$ for $L=s^{j-1}(M)\simeq\Sigma^{j-1}M$ and applying Proposition~\ref{propHomotopyTriple}. 
\end{proof}

\section*{Acknowledgements}
The authors thank Mikiya Masuda for sharing his idea to study torus actions in $j$-general position and for his hospitality during our visit to Osaka in 2018. 

\end{document}